\documentclass[a4paper,UKenglish,cleveref,autoref,thm-restate]{lipics-v2021}

\usepackage{float,cite,color,lineno}
\usepackage{tikz,pgfplots}
\usepackage{thmtools,thm-restate}
\usepackage{enumitem}
\usepackage{microtype}
\pgfplotsset{compat=1.17}

\tikzset{
	dot/.style = {circle, fill, minimum size=#1,
		inner sep=0pt, outer sep=0pt},
	dot/.default = 5pt
}
\usetikzlibrary{math}

\usepackage{regexpatch}
\makeatletter
\xpatchcmd\thmt@restatable{%
\csname #2\@xa\endcsname\ifx\@nx#1\@nx\else[{#1}]\fi
}{%
\ifthmt@thisistheone
\csname #2\@xa\endcsname\ifx\@nx#1\@nx\else[{#1}]\fi
\else
\csname #2\@xa\endcsname[{\bf\sffamily restated}]
\fi}{}{}
\makeatother

\newcommand{\NP}{{\sf NP}}

\newcommand{\ssi}{\subseteq_i}

\title{Determining the Complexity of Chromatic Sum in Classes Defined by a Set of Forbidden Graphs} 
\titlerunning{Chromatic Sum in Classes Defined by a Set of Forbidden Graphs}

\author{Cl\'ement Dallard}{Department of Informatics, University of Fribourg, Fribourg, Switzerland}{clement.dallard@unifr.ch}{https://orcid.org/0000-0002-9522-3770}{}
\author{Dani\"el Paulusma}{Department of Computer Science, Durham University, Durham, United Kingdom}{daniel.paulusma@durham.ac.uk}{0000-0001-5945-9287}{Supported by the Leverhulme Trust Grant RPG-2024-182.}
\author{Erik Jan van Leeuwen}{Department of Information and Computing Sciences, Utrecht University, Utrecht, The Netherlands}{e.j.vanleeuwen@uu.nl}{0000-0001-5240-7257}{}

\authorrunning{C.\ Dallard, D.\ Paulusma, and E.J.\ van Leeuwen}

\Copyright{Cl\'ement Dallard, Dani\"el Paulusma and Erik Jan van Leeuwen}

\ccsdesc[500]{Mathematics of computing~Graph theory}
\ccsdesc[500]{Theory of computation~Graph algorithms analysis}
\ccsdesc[500]{Theory of computation~Problems, reductions and completeness}

\keywords{complexity dichotomy, graph colouring, induced subgraph, subgraph}

\category{}
\relatedversion{}
\acknowledgements{The authors wish to thank Hans Bodlaender and Danny Hermelin for their helpful ideas towards the proofs of Theorem~\ref{t-c3} and~\ref{t-cw}.}

\hideLIPIcs  

\EventEditors{John Q. Open and Joan R. Access}
\EventNoEds{2}
\EventLongTitle{42nd Conference on Very Important Topics (CVIT 2016)}
\EventShortTitle{CVIT 2016}
\EventAcronym{CVIT}
\EventYear{2016}
\EventDate{December 24--27, 2016}
\EventLocation{Little Whinging, United Kingdom}
\EventLogo{}
\SeriesVolume{42}
\ArticleNo{23}

\makeatletter
\let\linenumbers\nolinenumbers

\nolinenumbers
\makeatother

\begin{document}
\maketitle

\begin{abstract}
The {\sc Chromatic Sum} problem asks, given a graph $G$ and an integer $k$, whether $G$ admits a colouring $c$ with sum $\sum_{v\in V}c(v) \leq k$. We study the complexity of {\sc Chromatic Sum} on graph classes defined by some set of forbidden graphs. First, we show that three known frameworks fully classify the complexity of {\sc Chromatic Sum} on ${\cal H}$-minor-free graphs and ${\cal H}$-topological-minor-free graphs for any set of graphs~${\cal H}$, and on ${\cal H}$-subgraph-free graphs for any finite set of graphs~${\cal H}$. To show this, we prove a new \NP-completeness result for {\sc Chromatic Sum} on certain subdivisions of planar subcubic graphs. Next, we consider other containment relations. We formalise a novel framework of problems that are \NP-complete for planar graphs as well as for graphs of bounded independence number. For every problem in this framework, we obtain an almost complete complexity classification on $H$-induced-minor-free graphs, $H$-induced-topological-minor-free graphs, and $H$-free graphs for every graph~$H$. We show that {\sc Chromatic Sum} belongs to this framework, as do several other problems. We also define a more fine-grained framework for the induced subgraph relation. We apply this to obtain a complete complexity classification for {\sc Chromatic Sum} on $H$-free graphs, as well as for several other problems. We justify the choice of this framework by proving that {\sc Chromatic Sum} is \NP-complete for graphs of clique-width at most~$3$. This result complements a known polynomial-time result for graphs of clique-width at most~$2$.
\end{abstract}

\section{Introduction}\label{s-intro}

Let $G=(V,E)$ be a graph.
A function $c:V\to \{1,2,\ldots\}$ is a {\it colouring} of $G$ if $c(u) \neq c(v)$ for every edge $uv\in E$.
The {\it sum} of $c$ is defined as $\sum_{v\in V}c(v)$. The {\it chromatic sum} of $G$ is the minimum sum over all colourings of $G$.
The {\sc Chromatic Sum} (or {\sc Sum Colouring}) problem, formulated as a decision problem, asks whether, given a graph $G$ and integer~$k$, the graph $G$ admits a colouring with sum at most~$k$. The {\sc Chromatic Sum} problem is motivated by applications in VLSI design~\cite{NicolosoSS99,Szkaliczki99,Supowit87} and scheduling~\cite{Bar-NoyBHST98,HalldorssonKS03}.

The study of {\sc Chromatic Sum} in the past decades has yielded various insights into its complexity, see e.g.~\cite{Bar-NoyBHST98,Bar-NoyK98,BodlaenderHL25,BonomoV09,BorodinIYZ12,Gavril72,GandhiHKS08,HalldorssonK02,Ja96,Ja97,Kubicka04,KubickaS89,Ma01,MalafiejskiGJK04,Marx05,NicolosoSS99,Salavatipour00,Salavatipour2003,Szkaliczki99,Supowit87}. Notably, {\sc Chromatic Sum} is \NP-complete~\cite{KubickaS89}, even on planar graphs~\cite{Ma01,HalldorssonK02} (even if they are bipartite and of maximum degree~$5$~\cite{MalafiejskiGJK04}, and even if they are cubic~\cite{Ma01}), interval graphs~\cite{Szkaliczki99,Marx05}, split graphs~\cite{Salavatipour00,Salavatipour2003}, penny graphs, unit square graphs, and unit disk graphs~\cite{BorodinIYZ12}. Yet, it is polynomial-time solvable on trees~\cite{KubickaS89}, graphs of bounded treewidth~\cite{Ja97}, proper interval graphs~\cite{NicolosoSS99}, $P_4$-free graphs~\cite{Ja97}, and several generalizations of $P_4$-free graphs~\cite{BonomoV09,Salavatipour2003}.

We consider the complexity of {\sc Chromatic Sum} for graph classes defined by a set of forbidden graphs. The way these graphs are not contained in the input graph is determined by a containment relation. Four prominent graph operations, vertex deletion (VD), edge deletion (ED), edge contraction (EC), and vertex dissolution (VDi), lead to ten non-trivial, well-known containment relations, summarized in Table~\ref{t-survey}. 
If a graph $G$ does not contain a graph $H$ as an induced minor, then $G$ is {\it $H$-induced-minor-free}. If ${\cal H}$ is a set of graphs, then $G$ is {\it ${\cal H}$-induced-minor-free} if $G$ is $H$-induced-minor-free for every $H\in {\cal H}$. If ${\cal H}=\{H_1,\ldots,H_p\}$ for some $p\geq 1$, then we may say that $G$ is {\it $(H_1,\ldots,H_p)$-induced-minor-free} instead. We define the same notion of freeness for the other nine containment relations from Table~\ref{t-survey}, except that we say that a graph is {\it $H$-free}, {\it ${\cal H}$-free}, or {\it $(H_1,\ldots,H_p)$-free} instead of $H$-induced-subgraph-free, ${\cal H}$-induced-subgraph-free, or $(H_1,\ldots,H_p)$-induced-subgraph-free, respectively.

Known results yield only partial insights into the complexity of {\sc Chromatic Sum} on classes defined by a set of forbidden graphs. Three widely applicable meta-classifications for graph problems on ${\cal H}$-minor-free graphs~\cite{RS86}, on ${\cal H}$-topological-minor-free graphs~\cite{RS86}, on ${\cal H}$-subgraph-free graphs~\cite{JMOPPSL25}, and on ${\cal H}$-free graphs~\cite{LR22,JMOPPSL25} may help to determine the complexity of {\sc Chromatic Sum} on those classes \emph{fully}. For the other classes determined by the containment relations in Table~\ref{t-survey}, such frameworks are currently lacking.

\begin{table}[h]
    \begin{center}
        \begin{small}
            \begin{tabular}{|l|c|c|c|c|l|}
                \hline
                \; Containment Relation         & \; VD \; & \; ED \; & \;  EC \; & \; VDi\\ 
                \hline
                \; minor                        & yes      & yes      & yes       & yes\\  
                \hline
                \; induced minor                & yes      & no       & yes       & yes\\  
                \hline
                \; topological minor            & yes      & yes      & no        & yes\\ 
                \hline
                \; induced topological minor \; & yes      & no       & no        & yes\\ 
                \hline
                \; contraction                  & no       & no       & yes       & yes\\
                \hline
                \; dissolution                  & no       & no       & no        & yes\\ 
                \hline
                \; subgraph                     & yes      & yes      & no        & no \\ 
                \hline
                \; induced subgraph             & yes      & no       & no        & no\\ 
                \hline
                \; spanning subgraph            & no       & yes      & no        & no\\ 
                \hline
                \; isomorphism                  & no       & no       & no        & no\\
                \hline
            \end{tabular}
        \end{small}
    \end{center}
    \caption{The ten graph containment relations in terms of their allowed graph operations. Two combinations are missing: ``no yes yes yes'', and  ``no yes no yes''. These  correspond to the minor and topological minor relations, respectively, if we allow a new operation that removes isolated vertices.
    We refer to Section~\ref{s-pre} for formal definitions of the graph operations.
    }\label{t-survey}
\end{table} 

\subparagraph*{Our Results}
We perform a systematic study into the complexity of {\sc Chromatic Sum} on graph classes defined by a set of forbidden graphs. This yields complete or almost-complete complexity classifications of {\sc Chromatic Sum} on all ten graph classes determined by the containment relations in Table~\ref{t-survey}. Our main conceptual contribution is the formalisation of two new complexity frameworks for $H$-free graphs, $H$-contraction-free graphs, $H$-induced-minor-free graphs, and $H$-induced-topological-minor-free graphs, which have wider applications.

We first summarize our results for {\sc Chromatic Sum}. A {\it subdivided claw} 
is obtained from the {\it claw} (the $4$-vertex star) after subdividing zero or more edges. The set ${\cal S}$ consists of all
disjoint unions of 
subdivided claws and paths.
The graphs $P_4$ and $P_1+P_3$ denote the 4-vertex path and the graph obtained from a $3$-vertex path and an isolated vertex, respectively.

\begin{restatable}{theorem}{SummaryThm}\label{t-summary}
    The following statements hold:
    \begin{enumerate}[label=(\roman*)]
        \item\label{t-summary-iso} For any finite set of graphs~${\cal H}$, {\sc Chromatic Sum} for ${\cal H}$-isomorphism-free graphs, ${\cal H}$-spanning-subgraph-free graphs, and ${\cal H}$-dissolution-free graphs is \NP-complete.
        \item\label{t-summary-minor} For a set of graphs~${\cal H}$, {\sc Chromatic Sum} for ${\cal H}$-minor-free graphs is polynomial-time solvable if ${\cal H}$ contains a planar graph, and is \NP-complete otherwise.
        \item\label{t-summary-topo} For a set of graphs~${\cal H}$, {\sc Chromatic Sum} for ${\cal H}$-topological-minor-free graphs is polynomial-time solvable if ${\cal H}$ contains a subcubic planar graph, and is \NP-complete otherwise.
        \item\label{t-summary-sub} For a finite set of graphs~${\cal H}$, {\sc Chromatic Sum} for ${\cal H}$-subgraph-free graphs is polynomial-time solvable if ${\cal H}$ contains a graph from ${\cal S}$, and \NP-complete otherwise.
        \item\label{t-summary-almost} For all but a finite number of graphs~$H$, {\sc Chromatic Sum} is \NP-complete on
        $H$-contraction-free graphs, $H$-induced-minor-free graphs, and $H$-induced-topological-minor-free graphs.
        \item\label{t-summary-h} For a graph~$H$, {\sc Chromatic Sum} for $H$-free graphs is polynomial-time solvable if $H$ is an induced subgraph of $P_4$ or $P_3+P_1$, and is \NP-complete otherwise.
    \end{enumerate}
\end{restatable}

\noindent
We discuss the ideas of the proof of Theorem~\ref{t-summary}; a formal argument appears in Section~\ref{s-summary-proof}.
The frameworks that underpin the statements~\ref{t-summary-iso}-\ref{t-summary-sub} are described in Section~\ref{s-overview}. In essence, statement~\ref{t-summary-iso} follows directly by the \NP-completeness of {\sc Chromatic Sum}~\cite{KubickaS89}.
Statements~\ref{t-summary-minor} and~\ref{t-summary-topo} follow by combining a meta-classification for ${\cal H}$-minor-free and ${\cal H}$-topological-minor-free graphs~\cite{RS86,JMOPPSL25} with the polynomial-time algorithm for {\sc Chromatic Sum} on graphs of bounded treewidth~\cite{Ja97} and its \NP-completeness on planar (cubic) graphs~\cite{Ma01,HalldorssonK02,MalafiejskiGJK04} (see also~\cite{GJKM02} and Section~\ref{s-1}).
For statement~\ref{t-summary-sub}, we apply a known meta-classification for ${\cal H}$-subgraph-free graphs~\cite{JMOPPSL25}.
To that end, we combine the known polynomial-time algorithm for {\sc Chromatic Sum} on graphs of bounded treewidth~\cite{Ja97} with a novel \NP-completeness result for {\sc Chromatic Sum}, on certain subdivisions of planar subcubic graphs (see Section~\ref{s-1}).

Statements~\ref{t-summary-almost} and~\ref{t-summary-h} rely on our newly formalised complexity frameworks.
Let the {\it independence number} of a graph $G$ be the size of a largest independent set in $G$. A class of graphs ${\cal G}$ has {\it bounded} independence number if there exists a constant~$c$ such that every $G\in {\cal G}$ has independence number at most~$c$. Consider the following conditions\footnote{The conditions are numbered this way to be consistent with existing literature, discussed in Section~\ref{s-overview}.} for a problem~$\Pi$:

\begin{description}[font=\bfseries,labelwidth=2em]
    \item [{\bf C4.}] $\Pi$ is \NP-complete for the class of planar graphs;
    \item [{\bf C6.}] $\Pi$ is \NP-complete for some class of graphs of bounded independence number.
\end{description}
A problem is a \emph{C46-problem} if it meets both condition C4 and C6.

\begin{restatable}{theorem}{AlmostThm}\label{t-almost}
    For all but a finite number of graphs~$H$, every C46-problem is \NP-complete on $H$-free graphs, $H$-contraction-free graphs, $H$-induced-minor-free graphs, and $H$-induced-topological-minor-free graphs. Moreover, in each of the open cases, $H$ is a planar $sP_1$-free graph for some $s\geq 1$.
\end{restatable}
We prove that {\sc Chromatic Sum} is a C46-problem. It is known to satisfy C4~\cite{Ma01,HalldorssonK02,MalafiejskiGJK04}. In Section~\ref{s-2}, we prove that {\sc Chromatic Sum} is \NP-complete for $(4P_1,2P_1+P_2,2P_2)$-free graphs. Hence, {\sc Chromatic Sum} satisfies C6 and is a C46-problem. Theorem~\ref{t-almost} may be of independent interest.
Namely, the following problems are also C46-problems, as they are \NP-complete on planar graphs and on co-bipartite graphs, which are $3P_1$-free: {\sc Bandwidth}~\cite{ParraS97,KloksKM99,GareyGJK78,Monien86}, {\sc Metric Dimension}~\cite{EpsteinLW15,DiazPSL17}, {\sc Identifying Code}~\cite{Foucaud15}, {\sc Locating-Dominating Set}~\cite{Foucaud15,Auger10,AugerCHL10},  {\sc Maximum Minimal Separator}~\cite{HanakaKKY21}, and {\sc Pathwidth}, which follows from combining results of \cite{ArnborgCP87,KloksKM99,MonienS88,ParraS97}; see Section~\ref{s-0}. Moreover, {\sc Colouring}~\cite{GSS76,KKTW01}, {\sc Precolouring Extension}~\cite{GSS76,KKTW01}, {\sc Fall Chromatic Number}~\cite{AELMMP26,LM20}, and
{\sc Fall Achromatic Number}~\cite{AELMMP26,LM20} 
are C46-problems, as they are \NP-complete on planar graphs and on $4P_1$-free graphs.

For the induced subgraph relation, we can get a full classification if we strengthen C6 and add two more conditions:

\begin{description}[font=\bfseries,labelwidth=2em]
    \item [{\bf C6'.}] $\Pi$ is \NP-complete for
          $H$-free graphs if $H\in \{C_3,C_4,C_5,4P_1,2P_1+P_2,2P_2\}$;
    \item [{\bf C7.}] $\Pi$ is polynomial-time solvable for $H$-free graphs if $H\in \{P_4,P_1+P_3\}$;
    \item [{\bf C8.}] $\Pi$ is \NP-complete for line graphs.
\end{description}
A problem is a \emph{C6'78-problem} if it meets conditions C6', C7, and C8.

\begin{restatable}{theorem}{FullyThm}\label{t-fully}
    For a graph~$H$, every C6'78-problem on $H$-free graphs is polynomial-time solvable if $H$ is an induced subgraph of $P_1+P_3$ or $P_4$, and \NP-complete otherwise.
\end{restatable}
This framework is different from an earlier proposed framework for $H$-free graphs (see Theorem~\ref{t-induced}) and appears (so far) to be more widely applicable.
In particular, it is known that, for example, {\sc Colouring}~\cite{KKTW01},
{\sc Precolouring Extension}~\cite{GPS14},
{\sc Fall Chromatic Number}~\cite{AELMMP26},
{\sc Fall Achromatic Number}~\cite{AELMMP26}, and
{\sc Critical Vertex}~\cite{PPR19}
are all C6'78-problems.
This indicates that Theorem~\ref{t-fully} may also be of independent interest.

We prove that {\sc Chromatic Sum} is also a C6'78-problem. As mentioned, we show that {\sc Chromatic Sum} is \NP-complete for $(4P_1,2P_1+P_2,2P_2)$-free graphs. Our new hardness result for subdivisions of planar subcubic graphs implies that {\sc Chromatic Sum} is \NP-complete for $(C_3,C_4,C_5)$-free graphs, and thus satisfies C6'. We show that {\sc Chromatic Sum} is polynomial-time solvable for $(P_1+P_3)$-free graphs (see Section~\ref{a-c7}). As Jansen~\cite{Ja96} proved the same for $P_4$-free graphs, {\sc Chromatic Sum} satisfies C7. Finally, Bar-Noy et al.~\cite{Bar-NoyBHST98} proved that {\sc Chromatic Sum} satisfies C8. Hence, it is a C6'78-problem.

We note that the proofs of Theorems~\ref{t-almost} and~\ref{t-fully} are straightforward and have already appeared implicitly in the literature. As such, our contribution is primarily conceptual, as it provides a unified framework that enables these results to be cited directly without the need to re-establish their proofs for each individual problem.

Finally, we consider that to cover part of C7, we could have stated the condition that $\Pi$ is polynomial-time solvable for every graph class of bounded clique-width. 
Indeed, $P_4$-free graphs are exactly the graphs of clique-width at most~$2$~\cite{CO00} and {\sc Chromatic Sum} is polynomial on $P_4$-free graphs~\cite{Ja96}. However, {\sc Chromatic Sum} is an exceptional problem:

\begin{theorem}\label{t-cliquewidth}
For an integer~$c$, {\sc Chromatic Sum} for graphs of clique-width at most~$c$ is polynomial-time solvable if $c\leq 2$ and \NP-complete if $c\geq 3$.
\end{theorem}
To show Theorem~\ref{t-cliquewidth} we prove in Section~\ref{a-distance} that {\sc Chromatic Sum} is \NP-complete for distance-hereditary graphs, which have clique-width at most~$3$~\cite{GolumbicR00}.

Finally, in Section~\ref{s-con} we give some directions for future work. 

\section{Preliminaries}\label{s-pre}

In what follows, we let $G=(V,E)$ denote a graph with vertex set $V$ and edge set $E$, where $E$ consists of unordered pairs of vertices of $V$.
For a set $S\subseteq V$, let $G[S]$ be the subgraph of $G$ induced by $S$. 
We write $G-S=G[V\setminus S]$. For a graph~$F$, we write $F\ssi G$ if $F$ is an induced subgraph of $G$.
For $u\in V$, the  set $N(u)=\{v\; |\; uv\in E\}$ denotes the {\it neighbourhood} of a vertex $u\in V$, and $|N(u)|$ is the {\it degree} of $u$.
If every vertex of $G$ has degree at most~$3$, then $G$ is {\it subcubic}.

The operations {\em vertex deletion} (VD) and {\em edge deletion} (ED) delete a vertex or an edge, respectively, from $G$. The \emph{edge contraction} (EC) operation, applied to an edge $uv$ of $G$, deletes the vertices $u$ and $v$ and replaces them by a new vertex made adjacent to precisely those vertices to which $u$ or $v$ were adjacent (without creating parallel edges). The {\it vertex dissolution} (VDi) operation can only be applied to a vertex $v$ of degree~2 whose two neighbours are not adjacent: it contracts one of the two edges incident with $v$ (which edge is irrelevant).

A set $M\subseteq E$ is a {\it matching} in $G$ if no two distinct edges in $M$ have a common vertex.
A set $K\subseteq V$ of pairwise adjacent vertices is a {\it clique} in $G$, and we say that $G$ is {\it complete} if $V$ is a clique. 
The {\it clique number} $\omega(G)$ of $G$ is the size of a largest clique in $G$.
A set $I\subseteq V$ of pairwise non-adjacent vertices is an  {\it independent set} in  $G$. We say that $G$ is {\it bipartite} if $V=A\cup B$ for two disjoint independent sets $A$ and $B$.
We say that $G$ is {\it complete multi-partite} if $V$ can be partitioned into independent sets $V_1,\ldots,V_r$, for some $r\geq 1$, such that for every $i\neq j$ there exists an edge between every vertex of $V_i$ and every vertex of $V_j$.

Recall that a colouring of $G$ is a mapping $c:V\to \{1,2,\ldots \}$ such that for every $uv\in E$, it holds that $c(u)\neq c(v)$.
We say that $c(u)$ is the {\it colour} of $u$.
The {\it colour classes} of $c$ are the sets of vertices assigned the same colour; they form a partition of $V$ into independent sets. 
The {\it chromatic number} $\chi(G)$ of $G$ is the smallest integer $k$ such that $G$ has a colouring $c:V\to \{1,\ldots,k\}$ that uses $k$ distinct colours.
Even the chromatic number of a tree~$T$ may be lower than the minimum number of colours needed to attain the chromatic sum~$T$~\cite{KubickaS89}.

For two vertex-disjoint graphs $G_1$ and $G_2$, we let $G_1+G_2=(V(G_1)\cup V(G_2),E(G_1)\cup E(G_2))$ denote the {\it disjoint union} of $G_1$ and $G_2$.  We write $rG$ for the disjoint union of $r$ copies of~$G$. A {\it linear forest} is a disjoint union of paths.

The {\it line graph} $L(G)$ of a graph $G$ has as vertices the edges from $E$, and there is an edge between two distinct vertices $e$ and $f$ of $L(G)$ if and only if $e$ and $f$ share a common end-vertex in $G$. It is well known and readily seen that every line graph is claw-free.
The {\it complement} $\overline{G}$ of $G$ is obtained by replacing each edge in $G$ with a non-edge, and vice versa.
A {\it co-component} of $G$ is a connected component of $\overline{G}$.
A graph is {\it co-bipartite} if its complement is bipartite.

Let  $C_r$ ($r\geq 3$), $P_r$ ($r\geq 1$) and $K_r$ ($r\geq 1$) denote the path, cycle and complete graph, respectively, on $r$ vertices.
For $r\geq 1$, let $K_{1,r}$ denote the star on $r+1$ vertices. Note that $K_{1,1}=K_2=P_2$ and $K_{1,2}=P_3$. Recall that the $4$-vertex star $K_{1,3}$ is the claw.

A graph class~${\cal G}$ has {\it bounded} treewidth if there exists a constant~$t$, such that every graph in ${\cal G}$ has treewidth at most~$t$ (we do not define treewidth, as we do not need it in our paper).

\section{Existing Frameworks} \label{s-overview}
We discuss severeal known frameworks for several of the containment relations in Table~\ref{t-survey}. 

\subparagraph{Subgraph Relation}
To explain a known meta-classification for the subgraph relation, we must define some new concepts.
The {\it $k$-subdivision} of a graph~$G$ is the graph obtained from~$G$ after subdividing each edge of $G$ exactly $k$ times.
For a graph class ${\cal G}$ and an integer~$k$, we let ${\cal G}^k$ be the class that consists of all the $k$-subdivisions of the graphs in ${\cal G}$ (in particular, ${\cal G}^0={\cal G}$).
We say that a graph problem $\Pi$ is \NP-complete {\it under edge subdivision of subcubic graphs} if for every integer $j \geq 1$, there exists an integer~$\ell \geq j$ such that the following holds:
if $\Pi$ is \NP-complete for the class ${\cal G}$ of subcubic graphs, then $\Pi$ is \NP-complete for ${\cal G}^{\ell}$.
The subgraph framework from \cite{JMOPPSL25} captures
every {\it C123-problem}, that is, a problem $\Pi$ is C123 if it satisfies the following three conditions:

\begin{description}[font=\bfseries,labelwidth=2em]
    \item [{\bf C1.}] $\Pi$ is polynomial-time solvable for every graph class of bounded treewidth;
    \item [{\bf C2.}] $\Pi$ is \NP-complete for the class of subcubic graphs; and
    \item [{\bf C3.}] $\Pi$ is \NP-complete under edge subdivision  of subcubic graphs.
\end{description}

\noindent
Recall that the set ${\cal S}$ consists of all non-empty disjoint unions of zero or more subdivided claws and paths.
Johnson et al.~\cite{JMOPPSL25} proved the following.

\begin{theorem}[\cite{JMOPPSL25}]\label{t-dicho}
    For a \emph{finite} set ${\cal H}$ of graphs, every C123-problem on ${\cal H}$-subgraph-free graphs is polynomial-time solvable if ${\cal H}$ contains a graph from ${\cal S}$, and \NP-complete otherwise.
\end{theorem}

\noindent
Johnson et al.~\cite{JMOPPSL25} proved that many problems are in fact C123, including well-known examples like {\sc Independent Set}, {\sc Vertex Cover}, {\sc Treewidth}, and {\sc Steiner Tree}. {\sc Colouring}, however, is not a C123-problem, as it is polynomial-time solvable on cubic graphs due to Brooks' Theorem, and thus it does not satisfy C2. Johnson et al.~\cite{JMPPSV23} consider the complexity of several further problems that do not satisfy C2. See also~\cite{BJMOPPSV25} for problems that do not satisfy C1, and~\cite{LMPPSSV24} for problems that do not satisfy C3.

\subparagraph*{Minor and Topological Minor Relation}
We first define the following two conditions for a problem $\Pi$, where C4' is a strengthening of C4.
\begin{description}[font=\bfseries,labelwidth=2em]
    \item [{\bf C4.}] $\Pi$ is \NP-complete for the class of planar graphs;
    \item [{\bf C4'.}] $\Pi$ is \NP-complete for the class of subcubic planar graphs.
\end{description}
Note that C4' indeed implies C4. A problem is a \emph{C14-problem} or a \emph{C14'}-problem if it satisfies conditions C1 and C4 respectively C1 and C4'. Due to Robertson and Seymour's graph minor theory~\cite{RS86} we have (as observed e.g.\ in~\cite{JMOPPSL25}):

\begin{theorem}[\cite{RS86}]\label{t-dicho2}
    For a set of graphs ${\cal H}$, every C14-problem on ${\cal H}$-minor-free graphs is polynomial-time solvable if ${\cal H}$ contains a planar graph, and \NP-complete otherwise.
\end{theorem}

\begin{theorem}[\cite{RS86}]\label{t-dicho3}
    For a set of graphs ${\cal H}$, every C14'-problem on ${\cal H}$-topological-minor-free graphs is polynomial-time solvable if ${\cal H}$ contains a subcubic planar graph, and \NP-complete otherwise.
\end{theorem}
Johnson et al.~\cite{JMOPPSL25} showed that many problems are also C14, C14', or both. A notable exception is {\sc Max-Cut}, which is polynomial-time solvable on planar graphs~\cite{Hadlock75}.
We briefly note that Theorem~\ref{t-dicho2} and~\ref{t-dicho3} immediately apply to {\sc Chromatic Sum} by known results~\cite{Ja97,Ma01}.

\subparagraph*{Induced Subgraph Relation}
A meta-classification for the induced subgraph relation is also known. To explain this meta-classification, we need to impose a new condition on a problem~$\Pi$ and some more terminology. The {\it line graph} of a graph $G$ has vertex set~$E(G)$, and there exists an edge between two vertices $e_1$ and $e_2$ if and only if $e_1$ and~$e_2$ have a common end-vertex in $G$. Let ${\cal T}$ be the class of line graphs of graphs of ${\cal S}$. Lozin and Razgon~\cite{LR22} proved that for any finite set of graphs~${\cal H}$, the class of ${\cal H}$-free graphs has bounded treewidth if and only if ${\cal H}$ contains a complete graph, a complete bipartite graph, a graph from~${\cal S}$ and a graph from ${\cal T}$.
Moreover, a class ${\cal G}$ has {\it bounded} treewidth if there exists a constant~$c$ such that every $G\in {\cal G}$ has treewidth at most~$c$; otherwise ${\cal G}$ has unbounded treewidth.
We now define the following condition for a problem~$\Pi$:

\begin{description}[font=\bfseries,labelwidth=2em]
    \item [{\bf C5.}] $\Pi$ is \NP-complete for every class of graphs of unbounded treewidth.
\end{description}

\noindent
We note that a C6'78-problem, or, more generally, any problem that satisfies C7 does not satisfy C5, because $P_4$-free graphs (and also $(P_1+P_3)$-free graphs) have unbounded treewidth.

A problem is a {\it C15}-problem if it meets conditions C1 and C5. The treewidth dichotomy of Lozin and Razgon~\cite{LR22} gives us the next theorem for C15-problems, as observed in~\cite{JMOPPSL25}.

\begin{theorem}[\cite{LR22}, see also~\cite{JMOPPSL25}]\label{t-induced}
    For a finite set of graphs~${\cal H}$, every C15-problem on ${\cal H}$-free graphs is polynomial-time solvable if ${\cal H}$ contains a complete graph, a complete bipartite graph, a graph from~${\cal S}$ and a graph from ${\cal T}$, and \NP-complete otherwise.
\end{theorem}

\noindent
Theorem~\ref{t-induced} could be viewed as a first meta-classification for the induced subgraph relation.
An example of a problem that satisfies the conditions of Theorem~\ref{t-induced} is {\sc Weighted Edge Steiner Tree}~\cite{BBJPPL21}, where the edges of the input graph have weights (in contrast, {\sc Edge Steiner Tree}  is a C123-problem~\cite{JMOPPSL25}).
However, in general, Theorem~\ref{t-induced} may be too restricted, as many graph problems can be solved in polynomial time on some graph class of unbounded treewidth (for instance, if the class of graphs under consideration has bounded clique-width or bounded mim-width).
Therefore, we need to search for other possible meta-classifications.

\subparagraph{Isomorphism, Spanning Subgraph, and Dissolution Relations}
Finally, we briefly consider ${\cal H}$-isomorphism-free, ${\cal H}$-spanning-subgraph-free, and ${\cal H}$-dissolution-free classes. Let $H$ be a graph. We note that every graph $G$ on at least $|V(H)|+1$ vertices is $H$-isomorphism-free and also $H$-spanning-subgraph-free. Hence, the isomorphism relation and spanning subgraph relation yield the same complexity results as for general graphs.

Moreover, a graph $G$ contains $H$ as a dissolution if $G$ is a subdivision of $H$. Hence, the vertex dissolution relation yields the same complexity results as for general graphs, a long as we may assume that we may restrict the input to graphs of sufficiently large maximum degree or that we may take as input a disjoint union of the original input graph $G$ and any constant-sized graph that is not a subdivision of $H$; such inputs are $H$-dissolution-free.

\section{The Proofs of Theorems~\ref{t-almost} and~\ref{t-fully}}\label{s-0}

We prove Theorem~\ref{t-almost} and show that {\sc Pathwidth} is a C46-Problem. Afterwards we prove Theorem~\ref{t-fully}. We restate both theorems below.

\AlmostThm*
\begin{proof}
Let $H$ be a graph, and let $\Pi$ be a C46-problem.
By C4, $\Pi$ is \NP-complete on planar graphs.
By C6, there exists an integer $s\geq 1$ such that $\Pi$ is \NP-complete on $sP_1$-free graphs.

Suppose first that $H$ is not planar.
Since the class of planar graphs is closed under taking induced subgraphs, contractions, induced minors, and induced topological minors, no planar graph contains $H$ in any of these four ways.
Therefore, the class of planar graphs is contained in the classes of $H$-free graphs, $H$-contraction-free graphs, $H$-induced-minor-free graphs, and $H$-induced-topological-minor-free graphs.
Thus $\Pi$ is \NP-complete on each of these classes.

Now suppose that $H$ is not $sP_1$-free.
The class of $sP_1$-free graphs is closed under taking induced subgraphs, contractions, induced minors, and induced topological minors, since none of these operations increases the independence number.
We get that no $sP_1$-free graph contains $H$ in any of the four corresponding ways.
As $\Pi$ is \NP-complete on $sP_1$-free graphs, it follows again that $\Pi$ is \NP-complete on each of the four graph classes under consideration.

Consequently, the only graphs $H$ not covered by the above arguments are those that are both planar and $sP_1$-free.
There are only finitely many such graphs.
Indeed, a planar graph has no clique of size $5$, while an $sP_1$-free graph has no independent set of size $s$.
By Ramsey's theorem, every planar $sP_1$-free graph has fewer than $R(5,s)$ vertices.
Hence, only finitely many graphs $H$ remain, and each of them is planar and $sP_1$-free for some $s\geq 1$.
\end{proof}

\noindent
We now show that {\sc Pathwidth} is a C46-Problem.
Recall that the {\sc Pathwidth} problem asks whether, given a graph $G$ and an integer $k$, the pathwidth of $G$ is at most $k$.

\begin{theorem}
{\sc Pathwidth} is a C46-problem.
\end{theorem}

\begin{proof}
For C4, we note that {\sc Pathwidth} is \NP-complete on planar (subcubic) graphs~\cite{MonienS88}.
For C6, we note that the pathwidth of a claw-free, AT-free graph is equal to its treewidth~\cite{ParraS97,KloksKM99}. Since co-bipartite graphs are claw-free and AT-free, and since {\sc Treewidth} (the problem of deciding whether a graph has treewidth at most~$k$) is \NP-complete on co-bipartite graphs~\cite{ArnborgCP87}, it follows that {\sc Pathwidth} is \NP-complete on co-bipartite graphs, which have independence number at most~$2$.
\end{proof}

\noindent 
Finally, we prove Theorem~\ref{t-fully}.

\FullyThm*
\begin{proof}
    Let $H$ be a graph and let $\Pi$ be a C6'78-problem.
    First suppose $H$ has a cycle.
    As $\Pi$ satisfies C6' and $2P_2\ssi C_r$ for every $r\geq 6$, we find that $\Pi$ is \NP-complete for $H$-free graphs.
    Now assume that $H$ is a forest.
    First suppose that $K_{1,3}\ssi H$.
    Note that the class of line graphs is a subclass of $K_{1,3}$-free graphs, and thus of $H$-free graphs.
    Hence, as $\Pi$ satisfies C8, we find that $\Pi$ is \NP-complete on $H$-free graphs.
    In the remaining case, $H$ is a linear forest.

    If $H$ contains at least four connected components, then $4P_1\ssi H$.
    Hence, as $\Pi$ satisfies C6', we find that $\Pi$ is \NP-complete on $H$-free graphs.
    Now suppose that $H$ consists of exactly three connected components.
    If one of the connected components of $H$ has an edge, then $2P_1+P_2\ssi H$.
    By C6', we find again that $H$ is \NP-complete on $H$-free graphs.
    Otherwise, $H=3P_1\ssi P_1+P_3$, and we find that $\Pi$ is polynomial-time solvable for $H$-free graphs due to C7.
    Now suppose that $H$ consists of two connected components.
    If both components of $H$ contain an edge, then $2P_2\ssi H$ and $\Pi$ is \NP-complete on $H$-free graphs due to C6'.
    Otherwise, $H=P_1+P_r$ for some $r\geq 1$.
    If $r\leq 3$, then $\Pi$ on $H$-free graphs is polynomial-time solvable due to C7, and otherwise it is \NP-complete due to C6'.
    Finally, suppose $H$ consists of one connected component, so $H=P_r$ for some $r\geq 1$.
    If $r\leq 4$, then $\Pi$ on $H$-free graphs is polynomial-time solvable by C7; otherwise, $2P_2\ssi H$, and it is \NP-complete by C6'.
\end{proof}

\section{The Proof of C3}\label{s-1}

In this section, we prove that {\sc Chromatic Sum} satisfies C3. Conditions C1 and C2 are known to hold~\cite{Ja97,Ma01}, and thus this is the only missing ingredient to apply Theorem~\ref{t-dicho} and completely classify the complexity of {\sc Chromatic Sum} on ${\cal H}$-subgraph-free graphs for any finite set ${\cal H}$.

\begin{theorem}\label{t-c3}
    Let ${\cal G}$ be the class of planar subcubic graphs.
    For every $\ell\geq 0$, {\sc Chromatic Sum} is \NP-complete for ${\cal G}^{2\ell}$, and thus satisfies C3.
\end{theorem}

\noindent
We note that this result is strictly stronger than (and implies, for $\ell=0$) \NP-completeness of {\sc Chromatic Sum} on planar subcubic graphs. Ma\l{}afiejski~\cite{Ma01} already proved that {\sc Chromatic Sum} is \NP-complete for planar subcubic graphs, but 
\href{https://doi.org/10.1016/S1571-0653(05)80080-0}{the electronic edition of~\cite{Ma01}} refers to a 1-page document that does not provide a proof. As such, we could not verify how much our proof resembles the proof of~\cite{Ma01}. Looking further, the construction in the \NP-completeness result for planar graphs by Halld{\'{o}}rsson and Kortsarz~\cite{HalldorssonK02} has maximum degree~$5$, but somewhat resembles our construction: we both reduce from {\sc Independent Set} on planar graphs and replace every edge by a triangle with pendants (we need one instead of three), but our reduction adds subdivisions (and must do so for our proof to work). The reduction of Ma\l{}afiejski et al.~\cite{MalafiejskiGJK04} in their \NP-completeness result for planar bipartite graphs of maximum degree~$5$ is substantially different, as it starts from {\sc 3-Dimensional Matching}.

\begin{proof}[Proof of Theorem~\ref{t-c3}]
    We may assume that $\ell > 0$; as ${\cal G}^2 \subseteq {\cal G}^0$, \NP-completeness for ${\cal G}^2$ implies \NP-completeness for ${\cal G}^0$.
    Mohar~\cite{Mo01} proved that {\sc Independent set} is \NP-complete on planar cubic graphs.
    Let $(G,k)$ be an instance of {\sc Independent Set}, where $G$ is planar and cubic.
    We first construct an auxiliary graph $\widehat G$ as follows.
    For each edge $uv\in E(G)$, delete $uv$, introduce four new vertices $p_{uv}$, $q_{uv}$, $r_{uv}$, and $s_{uv}$, and add the edges $up_{uv}$, $p_{uv}q_{uv}$, $q_{uv}v$, $p_{uv}r_{uv}$, $q_{uv}r_{uv}$, and $r_{uv}s_{uv}$.
    For every edge $uv\in E(G)$, we denote by $\widehat G_{uv}$ the subgraph of $\widehat G$ induced by $\{u,v,p_{uv},q_{uv},r_{uv},s_{uv}\}$.
    When clear from the context, we simply write $p$, $q$, $r$, and $s$ instead of $p_{uv}$, $q_{uv}$, $r_{uv}$, and $s_{uv}$.

    The graph $G'$ is obtained from $\widehat G$ by subdividing every edge exactly $2\ell$ times.
    For every edge $uv\in E(G)$, we denote by $G'_{uv}$ the subgraph of $G'$ obtained from $\widehat G_{uv}$ by this subdivision.
    We call $G'_{uv}$ the \emph{edge gadget} for $uv$.
    The six edges $up$, $pq$, $qv$, $pr$, $qr$, and $rs$ of $\widehat G_{uv}$ become six paths of length $2\ell+1$ in $G'_{uv}$.
    We call these paths the \emph{gadget paths} of $G'_{uv}$ and denote them by $P^{up}_{uv}$, $P^{pq}_{uv}$, $P^{qv}_{uv}$, $P^{pr}_{uv}$, $P^{qr}_{uv}$, and $P^{rs}_{uv}$, respectively.

    The auxiliary graph $\widehat G$ is planar and subcubic.
    Hence $G'$, being the $2\ell$-subdivision of $\widehat G$, belongs to $\mathcal{G}^{2\ell}$.
    \begin{figure}
        \centering
        \includegraphics[width=0.4\linewidth]{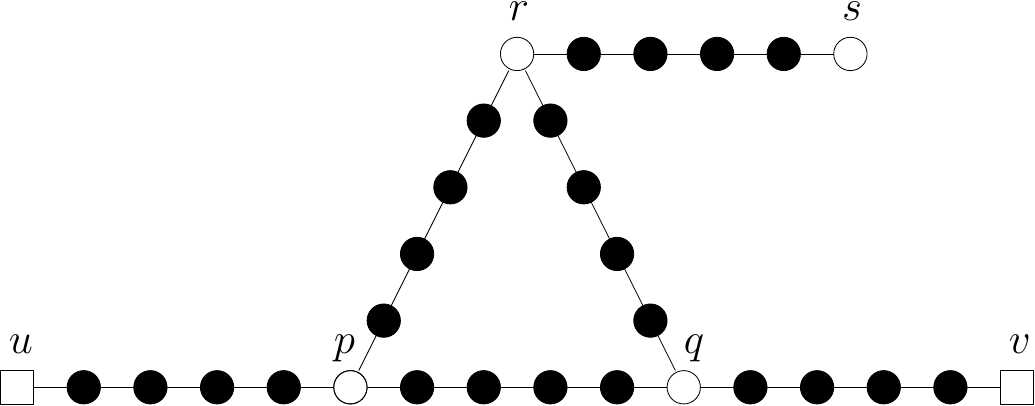}
        \hspace{0.08\linewidth}
        \includegraphics[width=0.4\linewidth]{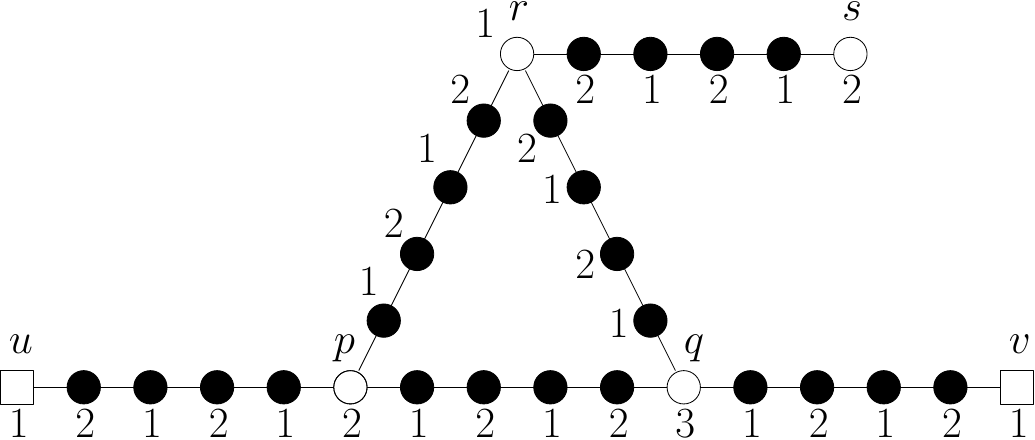}
        \\\bigskip
        \includegraphics[width=0.4\linewidth]{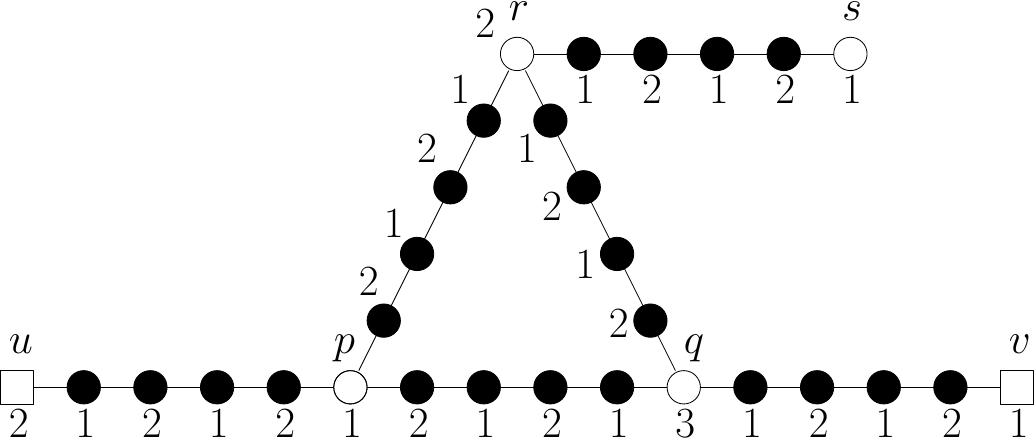}
        \hspace{0.08\linewidth}
        \includegraphics[width=0.4\linewidth]{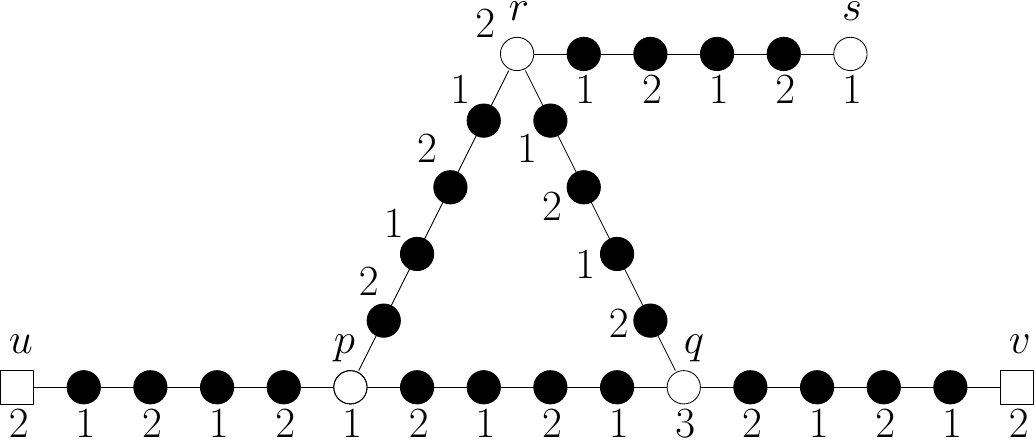}
        \caption{Construction and colourings of the edge gadget for the reduction of Theorem~\ref{t-c3}. \emph{Top left}: The edge gadget for an edge $uv \in E(G)$ for $\ell=2$. The vertices $u,v,p,q,r,s$ are indicated in white. \emph{Top right}: The colouring $c^1$ in this gadget. The sum of the colouring, except $u,v$, is $36+8$. \emph{Bottom left}: The colouring $c^2$ in this gadget. The sum of the colouring, except $u,v$, is $36+7$. \emph{Bottom right}: The colouring $c^4$ in this gadget. The sum of the colouring, except $u,v$, is $36+7$.}
        \label{fig:edgegadget}
    \end{figure}
We make the following claim:

    \begin{claim} \label{clm:lowerbase}
        For any colouring $c$ of $G'$ and for any odd integer $L$, the sum of the restriction of $c$ to any path $Q$ in $G'$ of length $L$ is at least $3(L+1)/2$. Moreover, this lower bound is only achieved if the colours of $c$ restricted to $Q$ are only $1$ and $2$.
    \end{claim}
    \begin{claimproof}
        Since $L$ is odd, the vertices of $Q$ can be partitioned into $(L+1)/2$ pairs of consecutive vertices.
        Each pair contributes at least $3$ to the sum, since its vertices are adjacent and receive distinct positive integer colours.
        Hence, the sum on $Q$ is at least $3(L+1)/2$.

        If equality holds, then the colours for every pair must sum to exactly $3$, and thus every vertex of $Q$ receives colour $1$ or $2$.
    \end{claimproof}
    The following is an immediate corollary of Claim~\ref{clm:lowerbase} and the fact that the gadget paths of $G'_{uv}$ have length $2\ell-1$ for every $uv \in E(G)$. Hence, the internal vertices of each gadget path induce a path of length $2\ell-1$.

    \begin{claim} \label{clm:lower}
        For any colouring $c$ of $G'$ and any $uv \in E(G)$, the sum of the restriction of $c$ to the internal vertices of the gadget paths of $G'_{uv}$ is at least~$18\ell$. Moreover, this lower bound is only achieved if the colours of $c$ restricted to those vertices are only~$1$ and~$2$.
    \end{claim}
    Now set $k' = (18\ell+7)m + 2n - k$, where $n$ and $m$ are the number of vertices and edges of $G$ respectively. Then $(G',k')$ is the constructed instance of {\sc Chromatic Sum}.

    We now show that $(G,k)$ is a yes-instance of {\sc Independent Set} if and only if $(G',k')$ is a yes-instance of {\sc Chromatic Sum}. To support this argument, we consider four possible colourings of the gadget $G'_{uv}$ for $uv \in E(G)$. Refer to Figure~\ref{fig:edgegadget}. Colouring $c^1$ assigns colour $1$ to vertices $u,v,r$, colour $2$ to vertices $p,s$, and colour $3$ to vertex $q$; the vertices of the gadget paths receive alternatingly colour $1$ and $2$. Note that since the ends of the gadget paths are assigned distinct colours and the gadget paths have odd length, this colouring is valid. Also note that, using Claim~\ref{clm:lower}, the sum of $c^1$, excluding vertices $u,v$, is $18\ell+8$. Colouring $c^2$ assigns colour $1$ to $v,p,s$, colour $2$ to $u,r$, and colour $3$ to $q$; the vertices of the gadget paths receive alternatingly colour $1$ and $2$. By the same argument, this colouring is valid. The sum of $c^2$, excluding vertices $u,v$, is $18\ell + 7$. We define colouring $c^3$ similarly as colouring $c^2$, except that we flip the colours of $u,v$ ($u$ receives colour $1$, $v$ colour $2$). This again yields a valid colouring with sum $18\ell+7$ (excluding $u,v$). Finally, colouring $c^4$ assigns colour $1$ to $p,s$, colour $2$ to $u,v,r$, and colour $3$ to $q$; the vertices of the gadget paths receive alternatingly colour $1$ and $2$. By the same argument, this colouring is valid. The sum of $c^4$, excluding $u,v$, is $18\ell + 7$.

    Now suppose that $(G,k)$ is a yes-instance of {\sc Independent Set}. Let $I$ be an independent set of $G$ of size at least $k$. We construct a colouring $c$ of $G'$ as follows. For every $u \in V(G)$, assign $u$ colour $1$ if $u \in I$ and colour $2$ otherwise. For every edge $uv \in E(G)$, we do as follows. If $u,v\not\in I$, then $u,v$ were assigned colour $2$ and we assign colouring $c^4$. Otherwise, we may assume that $v \in I$ and assign colouring $c^2$. This completes the description of $c$. Now consider the sum of $c$. Note that the vertices in $V(G)$ contribute at most $2n-k$ to the sum of $c$. The vertices of each edge gadget contribute $18\ell+7$ each by the above reasoning about $c^2$ and $c^4$. Hence, the sum of $c$ is at most $k'$, and $(G',k')$ is a yes-instance of {\sc Chromatic Sum}.

    For the converse, suppose that $(G',k')$ is a yes-instance of {\sc Chromatic Sum}. Let $c$ denote a colouring that has, in (lexicographic) order, minimum sum, minimum number of vertices of degree~$2$ with colour larger than~$2$, and minimum number of vertices of degree~$1$ or~$3$ with colour larger than~$2$. Note that the latter two conditions are not equivalent to $c$ having minimum sum and minimum number of vertices of colour larger than~$3$; the order is important. By assumption, the sum of $c$ is at most $k'$.

    For any vertex set $X \subseteq V(G')$, let $c(X) = \bigcup_{x \in X} \{c(x)\}$. We prove a series of claims to show that the vertices of $G$ coloured~$1$ by~$c$ form an independent set in $G$.

    \begin{claim} \label{clm:subset}
        The image of $c$ is a subset of $\{1,2,3\}$.
    \end{claim}
    \begin{claimproof}
        Suppose the image of $c$ contains a number $\alpha > 3$. Let $x$ be such that $c(x) = \alpha$. If $\{1,2,3\} \setminus c(N(x)) \not= \emptyset$, then we can recolour $x$ to a colour from $\{1,2,3\} \setminus c(N(x)) \not= \emptyset$ to obtain a colouring of $G'$ of strictly smaller sum than $c$. Hence, $c(N(x)) = \{1,2,3\}$ and $x$ has degree $3$. This argument implies that every vertex of degree at most~$2$ has colour at most~$3$ under $c$. Let $y \in N(x)$ be such that $c(y) = 3$. As $x$ has degree~$3$ and $\ell > 0$, it follows by the construction of $G'$ that $y$ has degree~$2$. Let $z \not= x$ be the other neighbour of $y$. Note that $z$ again has degree~$2$ by the construction of $G'$ and the fact that $\ell > 0$. Let $\beta \in \{1,2\}$ such that $c(z) \not= \beta$. Now recolour $y$ by colour $\beta$; this yields a valid colouring of $G'$ (note that $c(z) \not=\beta$ and $c(x) = \alpha > 3 > \beta$) of strictly smaller sum than~$c$, a contradiction.
    \end{claimproof}
    From now on, we use the result of Claim~\ref{clm:subset} implicitly, without referring to it.

    \begin{claim} \label{clm:ends}
        Let $uv \in E(G)$ and let $P$ be a gadget path of $G'_{uv}$ with ends $y,z$. If $c(x) = 3$ for some internal vertex of $P$, then $c(y)=c(z) \in\{1,2\}$.
    \end{claim}
    \begin{claimproof}
        Recall that $P$ has length $2\ell+1$. If $c(y) \not= c(z)$ or $c(y)=c(z)=3$, then we can recolour the internal vertices of $P$ alternatingly by colours $1$ and $2$ to yield a valid colouring $c'$. By Claim~\ref{clm:lowerbase}, the sum of the restriction of $c'$ to the internal vertices of $P$ achieves the lower bound of $3\ell$. Also by Claim~\ref{clm:lowerbase}, the lower bound cannot be achieved by $c$, as $c(x)=3$. Hence, the sum of $c'$ is strictly lower than the sum of $c$, a contradiction, and the claim follows.
    \end{claimproof}

    \begin{claim} \label{clm:threeneighb}
        Let $uv \in E(G)$ and let $P$ be a gadget path of $G'_{uv}$ with ends $y,z$. If $c(x) = 3$ for some internal vertex of $P$, then $c(N(x)) = \{1,2\}$.
    \end{claim}
    \begin{claimproof}
        By Claim~\ref{clm:ends}, $c(y)=c(z) \in \{1,2\}$. Let $a,b$ be the neighbours of $x$, such that $a$ is on the subpath $P^y$ of $P$ between $y$ and $x$ and $b$ is on the subpath $P^z$ of $P$ between $z$ and~$x$. Without loss of generality, we may assume that $P^z$ has even length, because $P$ has odd length by construction. Using Claim~\ref{clm:lowerbase} on the subpath of $P$ between $b$ and $z$ and the fact that $c(z) \in \{1,2\}$, we may assume that the colours used by $c$ on this subpath are $1$ and $2$ only, alternatingly, or we can obtain a colouring of strictly smaller sum than $c$. Hence, $c(b) \in \{1,2\} \setminus \{c(z)\}$. Similarly, the internal vertices of $P^y$ form a path of odd length. Using that $c(y) \in \{1,2\}$ and Claim~\ref{clm:lowerbase}, we may assume that the colours used by $c$ on this subpath are $1$ and $2$ only, alternatingly, or we can obtain a colouring of strictly smaller sum than $c$. Hence, $c(a)\in\{1,2\} \setminus\{3-c(y)\}$. As $c(y)=c(z)$, $c(N(x)) = \{1,2\}$.
    \end{claimproof}

    \begin{claim}\label{clm:nothrees}
        For each $uv \in E(G)$, if $x \in V(G'_{uv})$ has $c(x) = 3$, 
        then~$x$~has~degree~$1$~or~$3$~in~$G'$.
    \end{claim}
    \begin{claimproof}
        Let $x \in V(G'_{uv})$ be such that $c(x) = 3$. Since $G'$ is subcubic, $x$ has degree at most~$3$. Suppose that the degree of $x$ is~$2$.
        By the construction of $G'_{uv}$, $x$ is an internal vertex of a gadget path $P^1$. Let $y_1,z$ be the ends of $P^1$, where we use $z$ to denote a vertex of $P^1$ that is not in $V(G)$ and has degree~$3$, which exists by construction. By Claim~\ref{clm:ends}, $c(y_1)=c(z) \in \{1,2\}$.

        If $z$ has at least one neighbour of colour~$3$, recolour $z$ to~$3$ and recolour all vertices in $\Gamma^3 = \{a \in N(z) \mid c(a) = 3\}$ to colour $c(z)$. By Claim~\ref{clm:threeneighb}, every vertex in $N(\Gamma^3) \setminus \{z\}$ has colour not equal to $c(z)$. Hence, this yields a valid colouring $c'$. Moreover, either the sum of $c'$ is smaller than that of $c$ (if $|\Gamma^3| > 1$) or the sum has remained the same but the number of vertices of degree~$2$ with colour $3$ has decreased (if $|\Gamma^3| = 1$), a contradiction.

        Hence, $z$ has no neighbours of colour~$3$. Now let $P^1,P^2,P^3$ be the gadget paths incident on $z$. Recolour $z$ to have colour~$3$ and recolour the internal vertices of $P^1,P^2,P^3$ to $1$ and $2$, alternatingly. This yields a valid colouring $c'$ of $G'$, as $z$ receives colour~$3$. By Claim~\ref{clm:lowerbase} and because some vertex of $P^1$ was assigned colour~$3$ by $c$, either the sum of $c'$ is smaller than that of $c$ or the sum has remained the same, but the number of vertices of degree~$2$ with colour~$3$ has decreased, a contradiction. Hence, the degree of $x$ cannot be~$2$. 
    \end{claimproof}
Claim~\ref{clm:nothrees} implies that no vertex of $G'$ of degree~$2$ receives colour~$3$ under $c$. We henceforth use this fact implicitly.

    \begin{claim} \label{clm:notwothrees}
        For every $uv \in E(G)$, no distinct $x,y \in V(G'_{uv}) \setminus \{u,v\}$ have $c(x)=c(y)=3$.
    \end{claim}
    \begin{claimproof}
        Suppose that $x,y \in V(G'_{uv}) \setminus \{u,v\}$ have $c(x)=c(y)=3$. By Claim~\ref{clm:nothrees}, $x$ and $y$ have degree~$1$ or~$3$ in $G'$. Consider the vertices $p_{uv},q_{uv},r_{uv},s_{uv}$ of $G'_{uv}$. By construction, $p_{uv},q_{uv},r_{uv},s_{uv},u,v$ are the only vertices in $V(G'_{uv})$ that have degree~$1$ or~$3$ in $G'$. Moreover, using Claim~\ref{clm:lower} and $c(x)=c(y)=3$, it follows that the restriction of $c$ to $V(G'_{uv}) \setminus\{u,v\}$ uses at least $18\ell+8$ colours. But then we can recolour the vertices of $V(G'_{uv}) \setminus \{u,v\}$ according to $c^1$, $c^2$, $c^3$, or $c^4$ (use one of the latter three if $c(u)=3$ or $c(v)=3$) and obtain a valid colouring $c'$ of $G'$. Note that either the sum of $c'$ is lower than that of $c$, or the sums are equal and the number of vertices of degree~$2$ with colour~$3$ is equal (zero) while $c'$ colours less vertices of degree~$1$ or~$3$ of $G'$ with colour $3$ than $c$. This is a contradiction. Hence, no distinct $x,y \in V(G'_{uv}) \setminus \{u,v\}$ have $c(x)=c(y)=3$.
    \end{claimproof}

    \begin{claim} \label{clm:lowerseven}
        For every $uv \in V(G'_{uv})$, at least two vertices among $p_{uv},q_{uv},r_{uv}$ have colour at least~$2$ under~$c$ and at least one of them has colour~$3$ under~$c$. Hence, the sum of the restriction of $c$ to $V(G'_{uv}) \setminus \{u,v\}$ is at least $18\ell+7$.
    \end{claim}
    \begin{claimproof}
        Since the gadget paths $P^{pq},P^{pr},P^{qr}$ induce an odd cycle in $G'$, at least one of $p_{uv},q_{uv},r_{uv}$ has colour~$3$. Say $c(q_{uv})=3$. If $c(p_{uv})=c(r_{uv}) = 1$, then an internal vertex of $P^{pr}$ must have colour~$3$ under $c$, a contradiction to Claim~\ref{clm:nothrees}. Using Claim~\ref{clm:lower}, the sum of the restriction of $c$ to $V(G'_{uv}) \setminus\{u,v\}$ is at least $18\ell + 7$.
    \end{claimproof}

    \begin{claim} \label{clm:indep}
        For every $uv \in E(G)$, it is not true that $c(u)=c(v) = 1$.
    \end{claim}
    \begin{claimproof}
        Suppose that $c(u) = c(v) = 1$. Consider $G'_{uv}$ and the four vertices $p_{uv},q_{uv},r_{uv},s_{uv}$. We will argue that the sum of the restriction of $c$ to $V(G'_{uv})$ is at least $18\ell+8$. Indeed, by Claim~\ref{clm:nothrees}, no internal vertex of a gadget path of $G'_{uv}$ has been assigned colour~$3$. Hence, $c(p_{uv}) \not= 1$ and $c(q_{uv}) \not= 1$. Similarly, the path $P^{pq}$ ensures that $p_{uv},q_{uv}$ do not both receive colour~$2$ under $c$. By Claim~\ref{clm:notwothrees}, $p_{uv},q_{uv}$ do not both receive colour~$3$ under $c$. Assume that $c(p_{uv}) = 2$ and thus $c(q_{uv}) = 3$. Following similar reasoning as before, the path $P^{pr}$ ensures that $c(r_{uv}) \not= 2$. If $c(r_{uv}) = 3$, then we contradict Claim~\ref{clm:notwothrees}. Hence, $c(r_{uv}) = 1$. Then the path $P^{rs}$ ensures that $c(s_{uv}) \geq 2$, and thus by Claim~\ref{clm:lower}, the sum of the restriction of $c$ to $V(G'_{uv})$ is at least $18\ell+8$. The same arguments hold assuming that $c(p_{uv}) = 3$ and thus $c(q_{uv}) = 2$. Hence, the sum of the restriction of $c$ to $V(G'_{uv})$ is at least $18\ell+8$.

        Now let $u',u''$ be the neighbours of $v$ other than $u$ in $G$. Recolour $v$ to colour~$2$ and recolour the vertices $V(G'_{uv}) \setminus \{u,v\}$ according to $c^3$. Recolour $V(G'_{u'v}) \setminus \{u',v\}$ and $V(G'_{u''v}) \setminus \{u'',v\}$ according to $c^2,c^3,c^4$ depending on the colour of $u',u''$ (note that this is also possible if $c(u') = 3$ or $c(u'') = 3$). By Claim~\ref{clm:lowerseven} 
        because
        the sum of the restriction of $c$ to $V(G'_{uv})$ is at least $18\ell+8$, while $c^2,c^3,c^4$ only sum to $18\ell+7$ for the internal vertices of an edge gadget, this yields a valid colouring of $G'$ of strictly lower sum than $c$, a contradiction.
    \end{claimproof}
    Now let $I = \{u \in V(G) \mid c(u) = 1\}$.
    By Claim~\ref{clm:indep}, $I$ is an independent set of $G$.
    Hence, the sum of the restriction of $c$ to the vertices of $G$ is at least $|I|+2(n-|I|) = 2n-|I|$.
    The sum of the restriction of $c$ to the internal vertices of each edge gadget is at least~$18\ell+7$ by Claim~\ref{clm:lowerseven}.
    Recall that $k' = (18\ell+7)m + 2n - k$ and that the sum of $c$ is at most $k'$.
    So $2n-|I| + (18\ell+7)m \leq k' = (18\ell+7)m + 2n - k$. Hence, $|I| \geq k$.
    Thus, $(G,k)$ is a yes-instance of {\sc Independent Set}. The theorem follows.
\end{proof}

\section{Chromatic Sum on \texorpdfstring{$(C_5,4P_1,2P_1+P_2,2P_2)$}{(C5,4P1,2P1+P2,2P2)}-Free Graphs}\label{s-2}

For our next result, we reduce from {\sc Colouring}, which is known to be \NP-complete for this class~\cite{KKTW01} (a slightly different reduction was presented in~\cite{PPR19} to prove hardness for detecting a critical vertex; however, it does not yield any stronger result for {\sc Chromatic Sum}).

\begin{theorem}\label{t-kr}
    {\sc Chromatic Sum} is \NP-complete for $(C_5,4P_1,2P_1+P_2,2P_2)$-free graphs.
\end{theorem}

\begin{proof}
    We reduce from {\sc Colouring} restricted to
    $(C_5,2P_2,2P_1+P_2,4P_1)$-free graphs. As mentioned, Král et al.~\cite{KKTW01} proved that {\sc Colouring} is \NP-complete for this graph class.

    Let $(G,k)$ be an instance of {\sc Colouring}, where $G$ is a $(C_5,2P_2,2P_1+P_2,4P_1)$-free graph.
    Let $n = |V(G)|$, and let $K$ be a clique on $p$ new vertices, where $p=\frac{1}{2}n(n+1)$.
    Make every vertex of $K$ adjacent to every vertex of $G$, and let $G'$ be the resulting graph.
    Since every graph in $\{C_5,2P_2,2P_1+P_2,4P_1\}$ has no universal vertex, and $G$ is $(C_5,2P_2,2P_1+P_2,4P_1)$-free, the graph $G'$ is also $(C_5,2P_2,2P_1+P_2,4P_1)$-free.

    We set $q=\sum_{i=1}^p(k+i)$.
    We claim that $G$ is $k$-colourable if and only if
    $G'$ has chromatic sum at most $p+q$.
    First suppose $G$ is $k$-colourable.
    Among all $k$-colourings of $G$, choose one of minimum sum.
    Its sum is at most $1+\cdots+n=p$, as colouring the vertices injectively with colours $1,\ldots,n$ gives a valid colouring of $G$.
    We give the vertices of $K$ the colours $k+1,\ldots,k+p$.
    This gives a colouring of $G'$ with sum at most $p+q$.

    To prove the converse, let $c'$ be a colouring of $G'$ with sum at most $p+q$.
    Since $K$ is a clique complete to $V(G)$, the vertices of $K$ receive $p$ distinct colours, none of which are used on $V(G)$.
    Let $k_G$ be the number of colours used by the restriction of $c'$ to $G$.
    By permuting colours without increasing the sum, we may assume that the colour classes contained in $V(G)$ receive the lowest colours, since every colour class meeting $K$ is a singleton.

    Suppose that $k_G\geq k+1$, and so the contribution of the vertices of $K$ is at least $\sum_{i=1}^p(k+1+i) = q+p$.
    The contribution of the vertices of $G$ is at least $|V(G)|$. Hence, the total sum is at least $p+q+|V(G)|>p+q$, a contradiction.
    Therefore, $k_G\leq k$, and hence the restriction of $c'$ to $G$ uses at most $k$ colours.
    This proves the theorem.
\end{proof}

\section{Chromatic Sum on \texorpdfstring{$(P_3+P_1)$}{(P3+P1)}-Free Graphs}\label{a-c7}

In this section, we prove that {\sc Chromatic Sum} 
is polynomial-time solvable for $(P_3+P_1)$-free graphs.
For this, we will use the following characterization of paw-free graphs due to Olariu~\cite{Ol88}.
Here, the {\it paw} is the graph obtained by adding a pendant vertex to a triangle, that is, the graph with vertices $u_1,u_2,u_3,v$ and edges
$u_1u_2, u_2u_3, u_3u_1, u_1v$.

\begin{lemma}[\cite{Ol88}]\label{l-olariu}
    A graph $G$ is paw-free if and only if each connected component of $G$ is $C_3$-free or complete multi-partite.
\end{lemma}

\noindent
In fact, we will use an alternative but equivalent formulation of Lemma~\ref{l-olariu}.
We observe that $\overline{C_3}=3P_1$ and that $\overline{\mbox{paw}}=P_3+P_1$.
Hence, we can now formulate Lemma~\ref{l-olariu} as follows.

\begin{corollary}[\cite{Ol88}]\label{c-olariu}
    A graph $G$ is $(P_3+P_1)$-free if and only if, for each co-component $D$ of $G$, the subgraph $G[V(D)]$ is $3P_1$-free or a disjoint union of complete graphs.
\end{corollary}

\noindent
We can now prove the following result:

\begin{theorem}\label{t-p3p1}
    {\sc Chromatic Sum} is polynomial-time solvable for $(P_3+P_1)$-free graphs.
\end{theorem}

\begin{proof}
    Let $G$ be a $(P_3+P_1)$-free graph. Let $D$ be a co-component of $G$. By Corollary~\ref{c-olariu}, 
    $G[V(D)]$ is $3P_1$-free or a disjoint union of complete graphs.

    First suppose that $G[V(D)]$ is $3P_1$-free. This means that every colour class in every colouring of $G[V(D)]$ has size at most~$2$. Hence, to determine the chromatic sum of $G[V(D)]$, we must maximize the number of colour classes of size~$2$ in a colouring of $G[V(D)]$. We  can do this by computing, in polynomial time~\cite{Ed65}, a maximum matching $M$ in $\overline{G[V(D)]}$: the edges of $M$ form exactly the colour classes of size~$2$.
    
    Now suppose that $G[V(D)]$ is a disjoint union of complete graphs. In this case, any colouring that uses colours $1,\ldots,k$ to colour a connected component with $k$ vertices will give us the chromatic sum. This also takes polynomial time.

    The vertices of any other co-component $D'$ of $G$ are adjacent to all vertices of $D$. As such, the vertices of $D'$ must receive colours not used on $D$. Hence, the above implies that all we need to do is to order the co-components of $G$ as $D_1,\ldots,D_s$ for some $s\geq 1$, such that for every $i\in \{1,\ldots,s-1\}$, the chromatic sum of $G[V(D_i)]$ is at least the chromatic sum of $G[V(D_{i+1})]$. As the ordering takes polynomial time as well, this concludes the proof.
\end{proof}

\section{Proof of Theorem~\ref{t-summary}} \label{s-summary-proof}

For sake of completeness, we give a formal proof of Theorem~\ref{t-summary}.

\SummaryThm*
\begin{proof}
For statement~\ref{t-summary-iso}, we recall that the isomorphism relation and spanning subgraph relation yield the same complexity result as for general graphs. Since {\sc Chromatic Sum} is \NP-complete~\cite{KubickaS89}, {\sc Chromatic Sum} is \NP-complete for ${\cal H}$-isomorphism-free graphs and ${\cal H}$-spanning-subgraph-free graphs for any finite set of graphs~${\cal H}$. 

For the dissolution relation, we recall that the complexity is preserved from that of general graphs if, for any forbidden graph $H$, we can add any constant-sized graph that is not a subdivision of $H$ without changing the complexity. Note that, for any graph $H$, a clique on $|V(H)|+1$ vertices is not a subdivision of $H$. Moreover, the \NP-completeness of {\sc Chromatic Sum} is clearly maintained when adding a clique on $h \geq 4$ vertices by adding $\sum_{i=1}^h i = \binom{h+1}{2}$ to the budget. Hence, for any finite set of graphs~${\cal H}$, {\sc Chromatic Sum} for ${\cal H}$-dissolution-free graphs is \NP-complete.

For statement~\ref{t-summary-minor} and~\ref{t-summary-topo}, we recall that {\sc Chromatic Sum} is polynomial-time solvable on graphs of bounded treewidth~\cite{Ja97} and \NP-complete on planar graphs~\cite{Ma01,HalldorssonK02,MalafiejskiGJK04} and planar subcubic graphs~\cite{Ma01} (see also Theorem~\ref{t-c3}). Hence, {\sc Chromatic Sum} is a C14-problem and a C14'-problem, and we can apply Theorem~\ref{t-dicho2} and~\ref{t-dicho3} respectively.

For statement~\ref{t-summary-sub}, we recall that {\sc Chromatic Sum} is polynomial-time solvable on graphs of bounded treewidth~\cite{Ja97} and \NP-complete on planar subcubic graphs~\cite{Ma01} (see also Theorem~\ref{t-c3}), thus satisfying C1 and C2. We prove in Theorem~\ref{t-c3} that {\sc Chromatic Sum} satisfies C3. Hence, {\sc Chromatic Sum} is a C123-problem and we can apply Theorem~\ref{t-dicho}.

For statement~\ref{t-summary-almost}, we recall that {\sc Chromatic Sum} is \NP-complete on planar graphs~\cite{Ma01, HalldorssonK02,MalafiejskiGJK04} (see also Theorem~\ref{t-c3}), and thus satisfies C4. 
By Theorem~\ref{t-c3}, {\sc Chromatic Sum} satisfies C2 and C3. Hence, it is \NP-complete for $(C_3,C_4,C_5)$-free graphs.
Combining this result shows that {\sc Chromatic Sum} satisfies C6', and thus also C6. Hence, {\sc Chromatic Sum} is a C46-problem, and we can apply Theorem~\ref{t-almost}.

For statement~\ref{t-summary-h}, we recall that Bar-Noy et al.~\cite{Bar-NoyBHST98} proved that {\sc Chromatic Sum} is \NP-complete on line graphs, thus satisfying C8. 
We prove in Theorem~\ref{t-p3p1} that  {\sc Chromatic Sum} is polynomial-time solvable for $(P_3+P_1)$-free graphs. We recall that Jansen~\cite{Ja96} proved that  {\sc Chromatic Sum} is polynomial-time solvable for $P_4$-free graphs. Together, these two results imply that {\sc Chromatic Sum} satisfies C7. As mentioned before, {\sc Chromatic Sum} satisfies C6'. Hence, {\sc Chromatic Sum} is a C6'78-problem and we apply Theorem~\ref{t-fully}.
\end{proof}

\section{Chromatic Sum on Distance-Hereditary Graphs}\label{a-distance}

A graph $G$ is \emph{distance-hereditary} if $G$ can be obtained from a single vertex by the following operations: adding a pendant vertex, creating a true twin, or creating a false twin~\cite{BandeltM86}. Recall that vertices $u,v$ are \emph{true twins} if $N(x)=N(y)$ and $x,y$ are adjacent, and they are \emph{false twins} if $N(x)=N(y)$ and $x,y$ are not adjacent. Distance-hereditary graphs are known to have clique-width at most~$3$~\cite{GolumbicR00}.

\begin{theorem}\label{t-cw}
{\sc Chromatic Sum} is \NP-complete for distance-hereditary graphs, and thus for graphs of clique-width at most~$3$.
\end{theorem}
\begin{proof}
We first show that an intermediate problem is \NP-complete for distance-hereditary graphs. The {\sc Chromatic Sum with Lower Bounds} problem asks, given a graph $G$, a function $\ell:V(G) \rightarrow \{1,\ldots,|V(G)|\}$, and an integer $k$, whether $G$ admits a colouring $d$ with sum at most $k$ such that $d(v) \geq \ell(v)$ for each $v \in V(G)$. We prove the following, using a reduction that is essentially the same as one of Bodlaender et al.~\cite{BodlaenderHL25}.

\begin{claim}
{\sc Chromatic Sum with Lower Bounds} is \NP-complete for distance-hereditary graphs.
\end{claim}
\begin{claimproof}
We reduce from {\sc Precolouring Extension}, which asks, given a graph $G$, an integer $k$, a set $V_0 \subseteq V(G)$, and a function $c' : V_0 \rightarrow \{1,\ldots,k\}$, whether $G$ admits a colouring $c : V(G) \rightarrow \{1,\ldots,k\}$ such that $c(v) = c'(v)$ for each $v \in V_0$. It is known that {\sc Precolouring Extension} is \NP-complete for distance-hereditary graphs~\cite{BonomoDM09}. 

Let $(G,k,V_0,c')$ be an instance of {\sc Precolouring Extension} such that $G$ is a distance-hereditary graph. Without loss of generality, we may assume that $k \leq |V(G)|$. Let $X = |V(G)| \cdot k+1$. Let $\mu(v) = c'(v)$ if $v \in V_0$ and let $\mu(v) = |V(G)|$ otherwise; the intuition is that $\mu$ forms an upper bound on the colour to be assigned to $v$. Construct a graph from $G$ as follows. For every $v \in V(G)$ and every $i \in \{\mu(v)+1,\ldots,X\}$, add vertices $v_{i,1},\ldots,v_{i,X}$ to $G$ and make each adjacent to $v$. The resulting graph $G'$ is distance-hereditary, as the added vertices are pendant vertices. Now define $\ell:V(G') \rightarrow \{1,\ldots,X\}$ as follows: for every $v \in V(G)$, set $\ell(v) = c'(v)$ if $v \in V_0$ and $\ell(v) = 1$ otherwise; for every $v \in V(G)$ and every $i \in \{\mu(v)+1,\ldots,X\}$, set $\ell(v_{i,1})=\cdots=\ell(v_{i,X})=i$. Finally, let $k' = |V(G)| \cdot k + X \cdot \sum_{v \in V(G)} \sum_{i=\mu(v)+1}^{X} i$. Since $\ell(v) \leq X \leq |V(G')|$ for every $v \in V(G')$, $(G',\ell,k')$ is a valid instance of {\sc Chromatic Sum with Lower Bounds}. It also has polynomially bounded size, since $k \leq |V(G)|$.

Suppose that $(G,k,V_0,c')$ is a yes-instance of {\sc Precolouring Extension}. Let $c : V(G) \rightarrow \{1,\ldots,k\}$ be a colouring of $G$ such that $c(v) = c'(v)$ for each $v \in V_0$. Define $d : V(G') \rightarrow \{1,\ldots,|V(G'|\}$ as follows: for every $v \in V(G)$, set $d(v) = c(v)$; for every $v \in V(G)$ and every $i \in \{\mu(v)+1,\ldots,X\}$, set $d(v_{i,1})=\cdots=d(v_{i,X})=i$. As $d(v) = c(v) \leq \mu(v)$ for each $v \in V(G)$, $d$ is a valid colouring of $G'$. Also, for every $v \in V(G')$ it holds that $d(v) \geq \ell(v)$ by definition. Moreover, every $v \in V(G)$ contributes at most $k$ to the sum of $d$, while for every $v \in V(G)$ and every $i \in \{\mu(v)+1,\ldots,X\}$, the vertices $v_{i,1},\ldots,v_{i,X}$ contribute $iX$ to the sum of $d$. Hence, the sum of $d$ is at most $k'$, and $(G',\ell,k')$ is a yes-instance of {\sc Chromatic Sum with Lower Bounds}.

Suppose that $(G',\ell,k')$ is a yes-instance of {\sc Chromatic Sum with Lower Bounds}. Let $d$ be a colouring of $G'$ with sum at most $k'$. Observe that by the definition of $\ell$, the vertices of $V(G') \setminus V(G)$ contribute at least $X \cdot \sum_{v \in V(G)} \sum_{i=\mu(v)+1}^{X} i$ to the sum of $d$. If $d(v) \geq X$ for some $v \in V(G)$, then the sum of $d$ is larger than $k'$, a contradiction. Hence, $d(v) < X$ for every $v \in V(G)$. Suppose that $d(v) > \mu(v)$ for some $v \in V(G)$. Then by the definition of $\ell$ and the fact that $v_{d(v),1},\ldots,v_{d(v),X}$ exist (by $d(v)<X$) and are adjacent to $v$, it holds that $d(v_{d(v),1}),\ldots,d(v_{d(v),X}) > d(v)$. Then these vertices contribute an extra $X$ to the sum of $d$ and the vertices of $V(G') \setminus V(G)$ contribute at least $X + X \cdot \sum_{v \in V(G)} \sum_{i=\mu(v)+1}^{X} i > k'$ to the sum of $d$, a contradiction. Hence, $d(v) \leq \mu(v)$ for every $v \in V(G)$. Moreover, by the definition of $\ell$ and $\mu$, $d(v) = c'(v)$ for every $v \in V_0$. Hence, the restriction of $d$ to $V(G)$ demonstrates that $(G,k,V_0,c')$ is a yes-instance of {\sc Precolouring Extension}.
\end{claimproof}
We now prove the hardness of {\sc Chromatic Sum} on distance-hereditary graphs. Let $(G,\ell,k)$ be an instance of {\sc Chromatic Sum with Lower Bounds} such that $G$ is a distance-hereditary graph. Without loss of generality, we may assume that $k \leq 2\cdot |V(G)|^2$. For every $v \in V(G)$, add $k+1$ cliques $C_{v,1},\ldots,C_{v,k+1}$ of size $\ell(v)-1$ each to $G$ and make each adjacent to $v$. The resulting graph $G'$ is distance-hereditary, as the added vertices are true twins of pendant vertices of the vertices of $G$. Set $k' = k + (k+1) \cdot \sum_{v\in V(G)} (\ell(v)-1)$. Then $(G',k')$ is the constructed instance of {\sc Chromatic Sum}. This instance has polynomially bounded size, since $k \leq 2\cdot |V(G)|^2$ and $\ell(v) \leq |V(G)|$ for every $v \in V(G)$.

Suppose that $(G,\ell,k)$ is a yes-instance of {\sc Chromatic Sum with Lower Bounds}. Let $d$ be a colouring of $G$ of sum at most $k$ such that $d(v) \geq \ell(v)$ for every $v \in V(G)$. Let $c$ be the extension of $d$ by, for every $v \in V(G)$, colouring the vertices of each clique $C_{v,1},\ldots,C_{v,k+1}$ with colours $1,\ldots,\ell(v)-1$. These vertices contribute $(k+1) \cdot \sum_{v\in V(G)} (\ell(v)-1)$ to the sum of $c$. Hence, the sum of $c$ is at most $k'$, and $(G',k')$ is a yes-instance of {\sc Chromatic Sum}.

Suppose that $(G',k')$ is a yes-instance of {\sc Chromatic Sum}. Let $c$ be a colouring of $G'$ of sum at most $k'$. The vertices in the added cliques contribute at least $(k+1) \cdot \sum_{v\in V(G)} (\ell(v)-1)$ to the sum of $c$, because the cliques have size $\ell(v)-1$ each for every $v \in V(G)$. Suppose that $c(v) < \ell(v)$ for some $v \in V(G)$. Then in every clique among $C_{v,1},\ldots,C_{v,k+1}$, the colour $c(v)$ cannot be used and some vertex in each clique must be assigned a colour at least $\ell(v)$. Then these vertices contribute an extra $k+1$ to the sum of $d$ and the vertices of the added cliques contribute at least $k + 1 + (k+1) \cdot \sum_{v\in V(G)} (\ell(v)-1) > k'$ to the sum of $d$, a contradiction. Hence, $c(v) \geq \ell(v)$ for every $v \in V(G)$. Hence, the restriction of $c$ to $V(G)$ demonstrates that $(G,\ell,k)$ is a yes-instance of {\sc Chromatic Sum with Lower Bounds}. This completes the reduction.
\end{proof}

\section{Conclusions}\label{s-con}

We formalized several frameworks for graph containment and determined the computational complexity of {\sc Chromatic Sum} for graph classes defined by a set of forbidden subgraphs (see Theorem~\ref{t-summary}). In particular, we classified the complexity of {\sc Chromatic Sum} for
${\cal H}$-subgraph-free graphs for every set of finite graphs~${\cal H}$, and for $H$-free graphs for every graph~$H$. The situation where ${\cal H}$ is an arbitrary infinite set of forbidden subgraphs is still far from clear.
Similarly, we do not know the complexity of {\sc Chromatic Sum} for $(H_1,H_2)$-free graphs.
We leave these problems for future work, but we note that the classification for {\sc Colouring} on $(H_1,H_2)$-free graphs is also not yet resolved despite many partial results (see e.g.~\cite{GJPS17}). 

On a side note, even though the complexities of {\sc Chromatic Sum} and {\sc Colouring} 
are C6'78-problems, and thus their complexities 
coincide on $H$-free graphs, this is no longer the case if more than one induced subgraph is forbidden. 
For example, {\sc Colouring} is polynomial-time solvable for cubic graphs due to Brooks' Theorem, whereas {\sc Chromatic Sum} is \NP-complete for cubic graphs~\cite{Ma01} (see also Theorem~\ref{t-c3}). As another example, {\sc Colouring} is solvable in polynomial time on graphs of bounded cliquewidth~\cite{KR03}, while {\sc Chromatic Sum} is \NP-complete for graphs of clique-width~$3$
(Theorem~\ref{t-cliquewidth}). We also note that, while Theorem~\ref{t-summary} fully determines the complexity of {\sc Chromatic Sum} for $H$-subgraph-free graphs for any $H$, the complexity classification of {\sc Colouring} for $H$-subgraph-free graphs is still incomplete (see~\cite{EJLMP26,GPR15,JMPPSV23} for partial results).

We also note that the list of known C6'78-problems from Section~\ref{s-intro} consists of colouring problems only. It would be interesting to identify other types of C6'78-problems.

We also showed that there only exists a finite set of graphs $H$ for which the complexity of {\sc Chromatic Sum} is open on $H$-induced-minor-free graphs, $H$-topological-minor-free graphs, and $H$-contraction-free graphs. All these graphs $H$ are planar $4P_1$-free.
We note that {\sc Chromatic Sum} is \NP-complete on split graphs~\cite{Salavatipour00,Salavatipour2003}, and observe that split graphs are closed under edge contraction and vertex deletion. Hence, we may even assume that all the remaining cases involve a planar $4P_1$-free split graph.
Some of these cases might be solved by making use of existing characterizations. However, for several other cases of graphs~$H$, additional structural research is needed. We leave this for future research as well.

There also exist problems that are neither C14 nor C46'7, but that are still fully classified on $H$-free graphs. For example,
{\sc b-Chromatic Number} is polynomial-time solvable if $H\ssi P_4$ and \NP-hard otherwise~\cite{AELMMP26} (interestingly, it is not known if {\sc b-Chromatic Number} is a C46-problem, as its complexity on planar graphs is a well-known open problem). 
As another example,
{\sc Longest Path Contraction} is polynomial-time solvable on $H$-free graphs if $H\ssi P_2+P_4$, $P_1+P_2+P_3$, $P_1+P_5$ or $sP_1+P_4$ for some $s\geq 0$, and \NP-complete otherwise~\cite{KP20}. In fact, many other classical problems have different characterizations on $H$-free graphs, such as {\sc Independent Set}~\cite{GLMPPR24}, {\sc Matching Cut}~\cite{LMO25}, and {\sc Efficient Domination}~\cite{BM16,LPV18}.
It would be interesting to detect new conditions for meta-classifications.

\bibliography{main}

\end{document}